\documentclass [11pt] {article}
\usepackage{amsfonts}
\usepackage{amssymb}
\usepackage{amsmath}
\usepackage{times}
\usepackage{color}
\usepackage{graphicx}

\newcommand{\Hom}{\mathrm{Hom}}

\newtheorem{theorem}{Theorem}[section]
\newtheorem{definition}{Definition}[section]
\newtheorem{lemma}[definition]{Lemma}
\newtheorem{corollary}[definition]{Corollary}

\def\GG{\mathcal{G}}

\def\<{\langle}
\def\>{\rangle}

\def\to{\rightarrow}

\begin{document}

\title{Ramanujan graphings and correlation decay in local algorithms}
\author{{\'Agnes Backhausz ~~ Bal\'azs Szegedy ~~ B\'alint Vir\'ag}}

\maketitle

\abstract{Let $G$ be a large-girth $d$-regular graph and $\mu$ be a random process on the vertices of $G$ produced by a randomized local algorithm. We prove the upper bound $(k+1-2k/d)\Bigl(\frac{1}{\sqrt{d-1}}\Bigr)^k$ for the (absolute value of the) correlation of values on pairs of vertices of distance $k$ and show that this bound is optimal. The same results hold automatically for factor of i.i.d processes on the $d$-regular tree. In that case we give an explicit description for the (closure) of all possible correlation sequences. Our proof is based on the fact that the Bernoulli graphing of the infinite $d$-regular tree has spectral radius $2\sqrt{d-1}$. Graphings with this spectral gap are infinite analogues of finite Ramanujan graphs and they are interesting on their own right. 

\bigskip

\section{Introduction}

Randomized local algorithms are special type of parallelized algorithms that can be used to produce various 
important structures in graphs (independent sets, dominating sets, matchings, colorings, local samples, etc.) 
in constant running time (see \cite{Cso1},\cite{Cso2},\cite{EL}, \cite{GL}, \cite{GS}, \cite{LHS}, \cite{KG}}, \cite{NS}). 

Let $d\in\mathbb{N}$ be 
a fixed number. The input of the algorithm is a graph $G$ of maximal degree at most $d$. The first step 
of the algorithm puts labels on the vertices of $G$ independently from a probabiliy space $\Omega$. The second step 
evaluates a function $f$ (rule of the algorithm) at each vertex $v$ that depends on the isomorphism type of the labeled 
neighborhood of $v$ of radius $r$. (If $\Omega$ is an infinite probability space then $f$ is assumed to be measurable.)

In this paper we focus on the case of large-girth $d$-regular graphs. 
We give a rather explicit description for the (pointwise closure) of all possible correlation sequences in a local 
algorithm. 
Moreover, we obtain that  
the absolute value of the correlation of the values on two vertices of distance $k$ is at most 
$(k+1-2k/d)(d-1)^{-k/2}$ provided that the girth of the graph is at least $k+2r+2$. Surprisingly, the bound 
itself does not depend on the radius $r$ which is related to the complexity of the algorithm. 
%In particular we obtain that our upper bound can not be improved. 
Furthermore, we show that our upper bound is essentially optimal.

In our proof we make use of infinite measurable graphs called {\it graphings}. It turns out that many properties of local algorithms on large-girth $d$-regular graphs can be studied through the properties of a single object $B_d$ called {\it Bernoulli graphing}. For example our result on the correlation decay relies on the fact that $B_d$ is a {\it Ramanujan graphing}. Ramanujan graphs were introduced in the seminal paper \cite{LPS} by Lubotzky, Phillips and Sarnak. A $d$-regular graph is Ramanujan if its second largest eigenvalue (in the absolute value) is at most $2\sqrt{d-1}$. The analogue of the second largest eigenvalue can easily be defined for graphings by the spectral radius on the orthogonal complement of the constant function. In the infinite case $2\sqrt{d-1}$ is the smallest possible value. (Note (see \cite{K}) that this value is also equal to the spectral radius of the adjacency operator of the infinite $d$-regular tree.)
Ramanujan graphings (and more generally graphings with spectral gap) are interesting on their own right.
(Note that M. Ab\'ert has a different, more general definition for the notion of a Ramanujan graphing. \footnote{Personal 
communication.})

Last but not least, our results can also be interpreted from an ergodic theoretic point of view. 
As we will see in Section \ref{randomproc}, $d$-regular graphings can be used to produce random 
processes on the $d$-regular infinite tree $T_d$ that are invariant under the full automorphism group. 
Conversely, automorphism invariant processes can be used to produce $d$-regular graphings. Processes that come 
from the Bernoulli graphing are usually called {\it factor of i.i.d} processes \cite{LN}. Our result on the 
possible correlation sequences and the correlation decay naturally generalizes to ``Ramanujan processes'' 
(i.e., processes that come from
Ramanujan graphings).

To broaden the view on the topic, the last section discusses an interesting connection to the representation theory of ${\rm Aut}(T_d)$. There is a triple correspondence between spherical representations, correlation sequences, and invariant Gaussian processes on $T_d$. 

\bigskip

\noindent{\bf Acknowledgement.}~~The authors are thankful to Mikl\'os Ab\'ert for various prompting suggestions 
and to Viktor Harangi for useful remarks about the manuscript.

\medskip

\section{Ramanujan and Bernoulli graphings}

\begin{definition}\label{defgraphing}
Let $X$ be a Polish topological space and let $\nu$ be a probability
measure on the Borel sets in $X$. A {\em graphing} is a graph $\GG$
on $V(\GG)=X$ with bounded maximal degree and Borel measurable edge set $E(\GG)\subset X\times
X$ such that
\begin{equation}
\label{eq:degreeCond} \int_A e(x,B) d\nu(x) = \int_B e(x,A) d\nu(x)
\end{equation}
for all measurable sets $A,B\subseteq X$, where $e(x,S)$ is the
number of edges from $x\in X$ to $S\subseteq X$.
\end{definition}

Note that finite graphs are special graphings defined on finite probability spaces with uniform distribution.
Let $\GG$ be as in Definition \ref{defgraphing}. If
$f:X\rightarrow\mathbb{C}$ is a measurable function then we define
$\GG f$ by
\begin{equation}
\GG f(x)=\sum_{(x,v)\in E(\GG)}f(v).\end{equation}
A short calculation shows (see \cite{LHS}) that $\GG$ is a self-adjoint operator on $L^2(X)$ of norm at most $d$ where $d$ is the maximal degree in $\GG$.

To keep our notation simple, in this paper we will only consider $d$-regular graphings (every vertex has degree $d$). In this case we say that $\mathcal{M}=\mathcal{G}/d$ is the Markov operator corresponding to $\mathcal{G}$. We have that for $f\in L^2(X)$ the value of $(\mathcal{M}^k f)(x)$ is equal to the expected value of $f$ at the end of a random walk of length $k$ started at $x$. Furthermore if $H\subseteq X$ is a positive measure set then $$p_k(H):=(1_H,\mathcal{M}^k 1_H)/\nu(H)$$ is the probability that a random walk of length $k$ started at a random point of $H$ ends in $H$. 

Let $L^2_0(X)$ denote the subspace in $L^2(X)$ consisting of functions with integral equal to zero.
If $\GG$ is a $d$-regular graphing then the constant $1$ function is an eigenfunction of $\GG$ with eigenvalue $d$ and so its orthogonal complement $L^2_0(X)$ is invariant under the action of $\GG$.
We denote the norm of $\GG$ on $L^2_0(X)$ by $\rho(\GG)$. If $\rho(\GG)<d$ then we say that $\GG$ {\it has spectral gap.} Note that the spectral gap is closely related to return probabilities and mixing rates of random walks on $\GG$.

The graphing $\GG$ is called {\it ergodic} if there is no measurable connected component $S\subset X$ of $\GG$ such that $0<\nu(S)<1$. Graphings are typically not connected as abstract graphs so ergodicity is a good substitute for the notion of connectivity. It is easy to see that if a $d$-regular graphing $\GG$ has spectral gap then it has to be ergodic. Furthermore, an ergodic graphing is either a finite connected graph or its probability space has no atoms.
 
The following statement on $\varrho(\GG)$ is a modification  of well-known 
facts about finite graphs (see e.g. \cite[Theorem 7.1.]{FT}) for graphings.
%We will use the following modification of well-known facts about finite graphs for the graphing $\rho(\GG)$.

\begin{lemma}\label{formula} If $\GG$ is an arbitrary $d$-regular graphing then
$$\rho(\GG)=d~\sup_{H,k}~\Bigl|\frac{p_k(H)-\nu(H)}{1-\nu(H)}\Bigr|^{1/k}$$ where $H$ runs through all positive measure sets in $X$ and $k\in\mathbb{N}$.
\end{lemma}

\medskip

We prove the following theorem. 

\begin{theorem}\label{boppa} Let $\GG$ be a $d$-regular graphing on an atomless probability space $(X,\nu)$. Then $\rho(\GG)\geq 2\sqrt{d-1}$.
\end{theorem}

Motivated by the previous theorem we will use the following definition.

\begin{definition} A $d$-regular graphing $\GG$ is {\it Ramanujan} if $\rho(\GG)\leq 2\sqrt{d-1}$.
\end{definition}

\medskip

It follows from Theorem \ref{boppa} that a Ramanujan graphing $\GG$ is either a finite Ramanujan graph or $\rho(\GG)=2\sqrt{d-1}$.

We continue with the definition of the Bernoulli graphing of the $d$-regular tree. 
Let $T_d$ denote the $d$-regular infinite tree and let $T_d^*$ denote the version of $T_d$ in which a special vertex $o$  called {\it root} is distinguished. 
The set $Y=[0,1]^{T_d^*}$ is a probability space with the product measure. The group of root preserving automorphisms of $T_d^*$ is also acting on $Y$ by the permutation of the coordinates. This action is obviously measure preserving. We denote by $\Omega_d$ the space $Y/{\rm Aut}(T_d^*)$ with the inherited probability measure $\nu_d$.
We connect two elements in $\Omega_d$ by an edge if one can be obtained from the other by replacing the root to a neighboring vertex. The graph $B_d$ constructed this way is a $d$-regular graphing (see \cite{LHS}) that is called the Bernoulli graphing of $T_d$. Note that with probability one the connected component of a random element in $\Omega_d$ is isomorphic to $T_d$.

The following theorem seems to have been known for a while (see \cite{KT} or \cite{LN} 
Theorem 2.1.) We include a simple proof for completeness.

\begin{theorem}\label{bernram} For every $d\geq 2$ the Bernoulli graphing $B_d$ is Ramanujan.
\end{theorem}

To prove Theorem \ref{boppa} and Theorem \ref{bernram} we will need some preparation.
The next lemma is an easy consequence of the spectral theorem.

\begin{lemma}\label{spectral} Let $\mathcal{F}$ be a bounded, self-adjoint operator on the Hilbert space $\mathcal{H}$ and assume that $G$ spans $\mathcal{H}$. Then
$$\|\mathcal{F}\|=\sup_{v\in G}\Bigl(\limsup_{i\to\infty}\Bigl|\frac{(v,\mathcal{F}^iv)}{(v,v)}\Bigr|^{1/i}\Bigr).$$
\end{lemma}

Using this lemma we are ready to prove Lemma \ref{formula}.

\medskip

\noindent{\it Proof of Lemma \ref{formula}} For a positive measure set $H\subseteq X$ let $g_H=\frac{1}{\sqrt{\nu(H)}}1_H-\sqrt{\nu(H)}1_X$. We will use that $g_H\in L^2_0(X)$. Then for every $k\in\mathbb{N}$ we have that 
$$\frac{\rho(\GG)}{d}=\rho(\mathcal{M})\geq |(g_H,\mathcal{M}^k g_H)/(g_H,g_H)|^{1/k}=|(p_k(H)-\nu(H))/(1-\nu(H))|^{1/k}.$$ To verify the above calculation note that $\mathcal{M}$ is a self-adjoint operator and thus $(1_X,\mathcal{M}^k 1_H)=(\mathcal{M}^k1_X,1_H)=(1_X,1_H)=\nu(H)$. The other inequality follows from Lemma \ref{spectral} and the fact that functions of the form $g_H$ span the space $L^2_0(X)$. 
\hfill $\square $

\medskip

The next lemma is well known from probability theory \cite{W}.

\begin{lemma}\label{rwreturn} Let $S\subset T_d^*$ be a finite subset and let $r_k(S)$ denote the probability that a random walk started at the root ends in $S$. Then $\lim_{k\to\infty}r_{2k}(S)^{1/2k}=\frac{2\sqrt{d-1}}{d}$, if $S$ contains vertices at even distance from the 
root. Moreover,  $\lim_{k\to\infty}r_{2k+1}(S)^{1/(2k+1)}=\frac{2\sqrt{d-1}}{d}$, if $S$ contains vertices at odd distance.
\end{lemma}

\bigskip

\noindent{\it Proof of Theorem \ref{boppa}:}~Since $T_d$ covers every $d$-regular graph it is clear that for every positive measure set $H\subseteq X$ we have that $p_k(H)\geq r_k(o)$. The fact that $X$ is atomless implies that $\nu(H)$ can be arbitrary small in Lemma \ref{formula} and thus we obtain that $r_k(o)^{1/k}\leq\rho(\GG)/d$. By Lemma \ref{rwreturn} this completes the proof.
\hfill$\square $

\bigskip

\noindent {\it Proof of Theorem \ref{bernram}:}~It follows from Theorem \ref{boppa} that $\rho(B_d)\geq 2\sqrt{d-1}$. Thus by Lemma \ref{spectral} it remains to show that for some spanning set $G\subset L_0^2(\Omega_d)$ the inequality $\limsup_{i\to\infty}|(v,B_d^iv)|^{1/i}\leq 2\sqrt{d-1}$ holds whenever $v\in G$. Let $G$ be the set of all functions with $0$ integral and norm $1$ on $\Omega_d$ that depend only on the labels in a bounded neighborhood of the root. 

For $r\in\mathbb{N}$ let $S_r$ denote the neighborhood of the root in $T_d^*$ of radius $r$. Assume that $f\in G$ is a function which depends only on the labels in $S_r$. Using the notation of Lemma \ref{rwreturn} we claim that $c_k=|(\mathcal{M}^k f,f)|\leq r_k(S_{2r})$ where $\mathcal{M}=B_d/d$ is the Markov operator. This fact together with Lemma \ref{rwreturn} will complete the proof. 

To prove the claim observe that $c_k$ is equal to the correlation of the values of $f$ at the two endpoints of a random walk of length $k$ started at a random point $x\in X$.  Using the construction of $\Omega_d$ we lift the situation to the probability space $[0,1]^{T_d^*}$. By abusing the notation we assume that $f$ is defined on $[0,1]^{T_d^*}$ and it is invariant under ${\rm Aut}(T_d^*)$. For an element $\omega\in [0,1]^{T_d^*}$ and $v\in T_d^*$ let $g(v,\omega)$ denote the value of $f$ when the root is replaced to $v$.  (The fact that $g(v,\omega)$ is well-defined relies on the fact that $f$ is invariant under ${\rm Aut}(T_d^*)$.) The value of $c_k$ has the following description. We choose a random labeling $\omega$ of the vertices $T_d^*$ with $[0,1]$, start a random walk of length $k$ at the root of $T_d^*$ and on this probability space we take the correlation between $g(o,\omega)$ and $g(v,\omega)$ where $v$ is the endpoint of the walk. Conditioned on the fact that the random walk ends outside $S_{2r}$ it is clear that $g(o,\omega)$ and $g(v,\omega)$ are independent and so the correlation is $0$. It follows that the return probability to $S_{2r}$ is an upper bound for $c_k$.  
\hfill $\square $

\medskip

We finish this section with an observation on the spectral properties of $B_d$.
We will use a general fact about operators. Assume $\mathcal{F}$ is a bounded self-adjoint operator on a Hilbert space $\mathcal{H}$ , $f\in\mathcal{H}$ and $p\in\mathbb{R}[x]$. Furthermore let $P$ denote the projection valued measure corresponding to $\mathcal{F}$ used in the spectral theorem. Then  
\begin{equation}\label{specfact}
(f,p(\mathcal{F})f)=\int p~d\sigma_f
\end{equation}
where $\sigma_f(M)=\|P(M)f\|_2^2$ for $M\subset \mathbb R$ measurable. 
Note that if $\|f\|_2=1$ then $\sigma_f$ is a probability measure and we say that $\sigma_f$ is the spectral measure of $f$ corresponding to $\mathcal{F}$.

\begin{lemma}\label{confunc} Let $h:\Omega_d\rightarrow\{-1,1\}$ be the function such that $h(\omega)=-1$ if the label on the root is in $[0,1/2]$ and $h(\omega)=1$ otherwise. Then the spectral measure of $h$ corresponding to $B_d$ is the Plancherel measure of $T_d$, i.e. it is concentrated on $[-2\sqrt{d-1},2\sqrt{d-1}]$, and its density is the following: $\frac{d}{2\pi}\frac{\sqrt{4(d-1)-t^2}}{d^2-t^2}$ (it is also called the Kesten--McKay measure, see e.g. \cite{W}).
\end{lemma}

\begin{proof} Let $\sigma_h'$ be the spectral measure of $h$ corresponding to $B_d/d$. We use the proof of Theorem \ref{bernram} for the specific function $h$. The argument yields that $(h,(B_d/d)^kh)$ is equal to the return probability of a random walk of length $k$ started at the root. On the other hand by (\ref{specfact}) $(h,(B_d/d)^kh)$ is equal to the $k$-moment of $\sigma_h'$. This completes the proof. 
\hfill $\square $\end{proof}

\medskip

\begin{corollary}\label{confunccor} The set $\{\sigma_f|f\in L_0^2(\Omega)~,~\|f\|_2=1\}$ corresponding to $B_d$ is dense in the set of all probability measures on $[-2\sqrt{d-1},2\sqrt{d-1}]$ with respect to the weak topology.
\end{corollary}

\begin{proof} For the specific function $h$ defined in Lemma \ref{confunc} we have that the support $\sigma_h$ is the full interval $[-2\sqrt{d-1},2\sqrt{d-1}]$. The existence of one such function (using the spectral theorem) implies the statement.  Indeed, for small $\varepsilon$, the uniform measure on the interval $[x,x+\varepsilon]$ can be approximated by $\sigma_{P h /h(x)}$, where $P$ is the spectral projection to the interval $[x,x+\varepsilon]$.
\hfill $\square $\end{proof}

\section{Random processes on the tree}

\label{randomproc}
In this section we describe how to produce random processes on the tree $T_d$ from graphings. 
Furthermore, the correlation decay of the process can be bounded by a function of the largest eigenvalue of the graphing.

\begin{definition} Let $\GG$ be a $d$-regular graphing on the space $(X,\nu)$. We denote by $\Hom_0(T_d,\GG)$ the set of maps $\phi:V(T_d)\rightarrow X$ with the following two properties:
\begin{itemize} 
\item the image of an edge in $T_d$ is an edge in $\GG$,
\item $\phi$ is injective on the neighborhood of any vertex in $T_d$.
\end{itemize}
\end{definition}

The set $\Hom_0(T_d,\GG)$ is basically the set of covering maps from $T_d$ to $\GG$. 
Note that if $P$ is a path of length $k$ in $T_d$ then the image of $P$ under any $\phi\in\Hom_0(T_d,
\GG)$ is a non-backtracking walk on $\GG$.

\begin{lemma} There is a unique probability measure $\kappa$ on $\Hom_0(T_d,\GG)$ satisfying the following properties:
\begin{itemize}
\item for every vertex $v\in V(T_d)$ the distribution of the image of $v$ under a $\kappa$-random function is $\nu$,
\item $\kappa$ is invariant under the action of the automorphism group of $T_d$.
\end{itemize}
\end{lemma}

\begin{proof} We can uniquely bulid up this probability measure in the following way. Let us start with an arbitrary fixed vertex $v$ of $T_d$. By the first requirement the image of $v$ has distribution $\nu$. Once the image of $v$ is determined, say $\phi(v)=x$, the remaining vertices of $T_d$ have to be mapped to the connected component of $x$. The second requirement guarantees that $\phi$ has to be a randomly chosen covering of this connected component. Note that the graphing axioms imply that the first requirement holds for every vertex of $T_d$. 
\hfill $\square $\end{proof}

Using the previous lemma we can create probability distributions of $\mathbb{R}$-valued functions on the vertices of $T_d$ in the following way. A similar construction was used in \cite{MB}, see Lemmas 35 and 36. Let $f:X\rightarrow\mathbb{R}$ be a function in $L^2(X)$. Then the composition $f\circ\phi$ is an $\mathbb{R}$ valued process on the vertices of $T_d$ where $\phi$ is chosen randomly from $\Hom_0(T_d,\GG)$. We denote this process by $\mu_f$. An important special case is when $\GG=B_d$. In this case processes of the form $\mu_f$ are called {\it factor of i.i.d processes.} These processes can be equivalently defined without graphings in the following way. Let $f:[0,1]^{T_d^*}\rightarrow\mathbb{R}$ be a measurable function that is invariant under the natural action of ${\rm Aut}(T_d^*)$ on $[0,1]^{T_d^*}$. Then we construct a probability measure on $\mathbb{R}^{T_d}$ such that first we put uniform independent labels on the vertices of $T_d$ and then at each vertex $v$ we evaluate $f$ by putting the root on $v$. (Note that this is well-defined since $f$ is automorphism invariant.)

\begin{theorem}\label{cordec} Let $\GG$ be a Ramanujan graphing on $(X,\nu)$ and let $f:X\rightarrow\mathbb{R}$ be a function in $L^2(X)$. Then, in the process $\mu_f$, the absolute value of the correlation of the values on two vertices $v,w\in T_d$ is at most $$(k+1-2k/d)\Bigl(\frac{1}{\sqrt{d-1}}\Bigl)^k$$ where $k$ is the distance of $v$ and $w$.
\end{theorem}

The rest of the section is the proof of the above theorem. We imitate the proof from the paper \cite{ABLS} in the infinite setting.
The main idea is that the correlation decay in $\mu_f$ can be expressed in terms of non-backtracking random walks on $\GG$. 

We define a graphing $\GG^{(k)}$  on $(X,\nu)$; two vertices are connected in $\GG^{(k)}$ if and only if their distance is 
exactly $k$ in $\GG$. More precisely, we need weighted graphings. That is, instead of 
subsets of $X\times X$, we label the edges with nonnegative integers in a Borel measurable way. These will be the multiplicities of the edges in the graphing. Otherwise the definition is the same as the original one. If $v, w\in T_d$ with distance $k$, and $f\in L^2_0(X)$ with $||f||_2=1$, then by the definition of $\mu_f$ the reader can easily check that
\begin{equation}\label{correq}
{\rm corr}(\mu_f(v), \mu_f(w))=\biggl(f, \frac{\GG^{(k)}f}{d(d-1)^{k-1}}\biggr).\end{equation}

The arguments in the proof of Theorem 1.1. of \cite{ABLS} are valid for $d$-regular graphings as well. Therefore for $k\geq 1$ we have

\begin{equation}\label{correq2}
\frac{\GG^{(k)}}{d(d-1)^{k-1}}=\frac{1}{\sqrt{d(d-1)^{k-1}}}q_k \bigg(\frac{\GG}{2\sqrt{d-1}}\bigg),
\end{equation}

where  
\[q_k(x)=\sqrt{\frac{d-1}{d}}U_k(x)-\frac{1}{\sqrt{d(d-1)}}U_{k-2}(x); 
\quad U_k(\cos \vartheta )=\frac{\sin((k+1)\vartheta)}{\sin \vartheta}, \]
i.e., $U_k$ is the $k$th Chebyshev polynomial of the second kind for $k\geq 0$, and $U_{-1}\equiv 0$.

The spectral mapping theorem implies that if $\mathcal F: \mathcal H\rightarrow \mathcal H$ is a bounded self-adjoint operator on a Hilbert space $\mathcal{H}$, and 
$p$ is a polynomial, 
then $\|p(\mathcal F)\|\leq \max_{x\in [-\|\mathcal F\|, \|\mathcal F\|]} \vert p(x)\vert$.  Since $\GG$ is a Ramanujan graphing, its norm on $L^2_0$ is $2\sqrt{d-1}$.
%Korábban graphingra használtuk a rót, valahová le kéne írni, hogy ez a kettő ugyanaz?
Hence in our case this yields that 
\[\varrho\bigg(\frac{\GG^{(k)}}{d(d-1)^{k-1}}\bigg)\leq \max_{x\in [-1,1]} \frac{\vert q_k(x)\vert}{\sqrt{d(d-1)^{k-1}}}.\]

\medskip

We claim that~\[\max_{x\in [-1,1]} q_k(x)=\sqrt{d/(d-1)}(k+1-2k/d).\]
To see this let $T_k(\cos(\theta))=\cos(k\theta)$ be the defining equation for the $k$-th Chebyshev polynomial of the first kind. It is easy to see that \[\sqrt{d(d-1)}q_k(x)=(d-2)U_k(x)+2T_k(x).\]
Both $|U_k|$ and $|T_k|$ have their maximal values at $1$ and $U_k(1)=k+1~,~T_k(1)=1$. By substituting $1$ into $q_k$ we get the claim.

\medskip

We obtain that
\[\varrho\bigg(\frac{\GG^{(k)}}{d(d-1)^{k-1}}\bigg)\leq\Bigl(\frac{1}{\sqrt{d-1}}\Bigl)^k\bigg ( k+1-\frac{2k}{d}\bigg).\]
This finishes the proof. \hfill $\square $

\medskip

\section{Randomized local algorithms}

As it was described in the introduction, a randomized local algorithm produces a random labeling of the vertices of a bounded degree graph using an initial i.i.d labeling and a local rule denoted by $f$. To give a precise definition we will need the following notation.

Let $S$ be an arbitrary set. An $S$-labeled graph $G$ is a graph together with a function $h:V(G)\rightarrow S$. A rooted graph is a graph in which a special vertex denoted by $o$ and called root is distinguished.  Assume that $r,d\in\mathbb{N}$. We denote by $\mathcal{N}(r,d,S)$ the set of isomorphism types of rooted, $S$-labeled graphs of maximum degree $d$ such that each vertex is of distance at most $r$ from the root. (Isomorphisms are assumed to be root and label preserving.)

A rule of radius $r$ and degree $d$ is a function $f:\mathcal{N}(r,d,S)\rightarrow S_2$ where $S_2$ is some set. 
Assume that $G$ is a graph of maximal degree at most $d$ and that $h:V(G)\rightarrow S$ is some labeling. Then we can use $f$ to produce a new labeling $h_2:V(G)\rightarrow S_2$ such that $h_2(v)$ is equal to the value of $f$ on the $S$-labeled rooted neighborhood of radius $r$ of $v$ where the root is placed on $v$. We denote the labeling $h_2$ by $h^f$.

\begin{definition} A randomized local algorithm of radius $r$ and degree $d$ is given by a measurable function $f:\mathcal{N}(r,d,\Omega)\rightarrow L$ (called rule of the algorithm) where $\Omega$ is a probability space and $L$ is a measure space. 
The input of the algorithm is a graph of maximal degree at most $d$ and the output is the random labeling $h^f$ where $h$ is a labeling of $G$ with independent, random elements from $\Omega$. 
\end{definition}

Note that local algorithms can also be computed on infinite graphs if they have bounded maximum degree. 
The next example produces independent sets in graphs \cite{AS}.

Let $f:\mathcal{N}(1,d,[0,1])\rightarrow \{0,1\}$ be the rule such that the value of $f$ is $1$ if and only if the label on the root is the smallest among all labels. It is clear that if $h:V(G)\rightarrow [0,1]$ is an arbitrary injective function then the support of $h^f$ is an independent set in $G$. Since random labelings $h:V(G)\rightarrow [0,1]$ are injective with probability $1$ we have that the local algorithm with rule $f$ produces a random independent set with probability $1$.

\medskip

In the rest of this section we focus on the case when $G$ is a $d$-regular graph with girth more than twice  the radius of $f$. In this case it is enough to define $f$ on $\Omega$-labeled versions of the neighborhood $S_r$ of the root in $T_d^*$ of radius $r$. In other words we can assume that $f$ is a function of the form $f:\Omega^{S_r}\rightarrow L$ that is invariant under the automorphisms of $S_r$. 

We can also represent $f$ as a function $g:\Omega_d\rightarrow L$ on the vertex set of the Bernoulli graphing $B_d$. Let $\phi:[0,1]\rightarrow\Omega$ be an arbitrary measure preserving map and $\psi:\Omega_d\rightarrow\Omega^{S_r}$ be the map defined by deleting the vertices outside $S_r$ and taking the $\phi$ images of the original labels. Let $g=f\circ\psi$. 
It is clear that the process $\mu_g$ on $T_d$ is the same as the process produced by the local algorithm on $T_d$ with rule $f$.
On the other hand if $G$ is any $d$-regular graph of girth at least $2(k+r+1)$ then the distribution of the local algorithm in any ball of radius $k$ is the same as its distribution on $T_d$ in a similar ball. 
It follows for example that to analyze local properties (such as correlation decay) of local algorithms in the large-girth setting, it is enough to consider the algorithm on the tree $T_d$. This creates the connection between Bernoulli graphings and local algorithms. As a corollary of Theorem \ref{cordec} and Theorem \ref{bernram} we obtain the following.

\begin{theorem} Assume that $r,k\in\mathbb{N}$ and that $f\in L^2(\Omega^{S_r})$ is invariant under ${\rm Aut}(S_r)$. Let $G$ be a $d$-regular graph of girth at least $k+2r+2$. If $v,w$ are two vertices of distance $k$ in $G$ then the absolute value of the correlation of the values on $v$ and $w$ in the local algorithm given by $f$ is at most $(k+1-2k/d)(d-1)^{-k/2}$.
\end{theorem}

\section{Characterization of correlation sequences}

\medskip

Our goal is to give an algebraic characterization (up to closure with respect to pointwise convergence) for possible correlation sequences in factor of i.i.d processes.  We return to the proof of Theorem \ref{cordec}. Let us apply (\ref{specfact}) to $f$ and $\mathcal{G}/(2\sqrt{d-1})$ in the calculation. We obtain that the value of (\ref{correq}) for two vertices of distance $k$ is equal to
\[\int_{[-1,1]} \frac{1}{\sqrt{d(d-1)^{k-1}}}q_k~d\sigma_f\]
where $\sigma_f$ is the spectral measure of $f$ with respect to $\mathcal{G}/(2\sqrt{d-1})$.
The next theorem follows immediately from Corollary \ref{confunccor}.

\begin{theorem}\label{corchar} Let $X_d$ denote the set of all sequences with 
$x_k=\int d^{-1/2}(d-1)^{(1-k)/2} q_k~d\eta$ where $\eta$ is a probability measure on $[-1,1]$.
Then the closure of possible correlation sequences in factor of i.i.d processes is equal to $X_d$.
\end{theorem}

\bigskip

We finish with an example for a local algorithm on the tree $T_d$. We start with the intitial i.i.d labeling $\{X_u\}_{u\in T_d}$ where $X_u=1$ with probability $1/2$ and $X_u=-1$ otherwise. 
For $r\geq 0$ denote by $S_r(v)$ the neighborhood of radius $r$ around $v\in T_d$. Note that $$|S_r(v)|=1+d\frac{(d-1)^r-1}{d-2}$$ for $r\geq 1$.

For every $w\in T_d$ we define the random variable $Y_w=\frac{1}{\sqrt{|S_r(w)|}}\sum_{u\in S_r(w)} X_u$. 
It is clear that $\{Y_w\}_{w\in T_d}$ is the output of a local algorithm.
Furthermore we have that
\begin{multline*}
{\rm corr}(\mu_f(v), \mu_f(w))=\frac{1}{|S_r(v)|}{\rm cov}\Bigg(\sum_{u\in S_r(v)} X_u, \sum_{u\in S_r(w)} X_u\Bigg)\\= 
\frac{1}{|S_r(v)|}{\rm var}\Bigg(\sum_{u\in S_r(v)\cap S_r(w)}X_u\Bigg)=\frac{|S_r(v)\cap S_r(w)|}{|S_r(v)|}=\frac{d(d-1)^{r-k/2}-2}{d(d-1)^r-2}
\end{multline*} if $k$ is even and $k<2r$.   
For the last equation we use the fact that $S_r(v)\cap S_r(w)$ is equal to $S_{r-{k/2}}(z)$ where $z$ is the middle point of of the path connecting $v$ and $w$.
Now, as $r$ goes to infinity, the lower bound for the correlation converges to $(d-1)^{-k/2}$.

For odd $k$ with $k\leq 2r+1$ we have two points in the middle and so
\[\frac{|S_r(v)\cap S_r(w)|}{|S_r(v)|}=\frac{2[1+(d-1)+\ldots+(d-1)^{r-(k+1)/2}]}{|S_r(v)|}=2\frac{(d-1)^{r-(k-1)/2}-1}{d(d-1)^r-2}.\] 
This converges to $\frac2d(d-1)^{-\frac{k-1}{2}}$ as $r\rightarrow \infty$.

This shows that the correlation decay is close to be optimal in this simple example.

\section{Semi-definite functions, spherical representations and Gaussian processes}

The goal of this section is to show how correlation sequences of invariant processes on $T_d$ can be viewed from a representation theoretic perspective. Moreover, every such correlation sequence produces a unique invariant Gaussian process on $T_d$ which is interesting on its own right.

Let $\mathcal{P}_d$ denote the set of all positive semi-definite functions $p:T_d\times T_d\rightarrow [-1,1]$ such that $p(v,v)=1$ and the value of $p(v,w)$ depends only on the distance of $v$ and $w$ for every pair $v,w\in T_d$.
It is clear that the correlation structure of an arbitrary (real valued) invariant process on $T_d$ is an element in $\mathcal{P}_d$.
On the other hand any element of $\mathcal P_d$ defines a symmetric representation of $T_d$ in some real 
Hilbert space. To be more precise, there is a function $\phi$ from $T_d$ to some separable Hilbert space 
$\mathcal{H}$ such that $p(v,w)=(\phi(v),\phi(w))$ and that $\{\phi(v)|v\in T_d\}$ generates $\mathcal{H}$. It is clear that $\phi$ is unique up to orthogonal transformations and that there is an orthogonal representation $\psi:{\rm Aut}(T_d)\rightarrow O(\mathcal{H})$ with the property that $\phi(\alpha(v))=\psi(\alpha)(\phi(v))$ for every $v\in T_d$ and $\alpha\in{\rm Aut}(T_d)$. In particular $\phi(o)$ is fixed under ${\rm Aut}(T_d^*)$ where $o$ is any distinguished root in $T_d$. Such representations of ${\rm Aut}(T_d)$ are called {\it spherical} in the literature. In other words, a representation of ${\rm Aut}(T_d)$ is spherical if the subgroup ${\rm Aut }(T_d^*)$ has a fixed vector $x$ of length $1$ such that the images of $x$ under ${\rm Aut}(T_d)$ generate the underlying Hilbert space. It is clear that each spherical representation $\psi$ of ${\rm Aut}(T_d)$ gives rise to an element in $\mathcal{P}_d$ by $p(v,w)=(\psi(\alpha_v)(x),\psi(\alpha_w)(x))$ where $\alpha_v(o)=v, \alpha_w(o)=w$. (The spherical property guarantees that $p$ is well-defined.) This construction yields a one to one correspondence between spherical representations and elements in $\mathcal{P}_d$. 

To complete the picture, for every $p\in\mathcal{P}_d$ we construct an invariant process with correlation structure $p$.
The most natural choice is an infinite dimensional Gaussian distribution $\gamma_p$ on $\mathbb{R}^{T_d}$ with correlation structure $p$. The uniqueness of $\gamma_p$ guarantees that it is an invariant under ${\rm Aut}(T_d)$.
If $\mu$ is any other invariant process with the same correlation structure $p$ then we can also obtain $\gamma_p$ by the central limit theorem in the following way. Assume that $\mu$ is already normalized in a way that it has zero expectation and variance (at each vertex) equal to $1$. Let $[\mu]_n$ denote the distibution of $q_1+q_2+\dots+q_n$ where each $q_i$ is an independent element of $\mathbb{R}^{T_d}$ chosen with distribution $\mu$. It is clear that the weak limit of $n^{-1/2}[\mu]_n$ is a Gaussian process with correlation structure $p$ and thus it is equal to $\gamma_p$.
If in particular $\mu$ is a factor of i.i.d process then so is $[\mu]_n$ for every $n$. 
It follows that in this case $\gamma_p$ is a weak limit of factor of i.i.d processes. Therefore we obtain the following.

\begin{corollary}
A Gaussian process is a limit of factor of i.i.d. processes if and only if its correlation decay is as in 
Theorem \ref{corchar}.
\end{corollary}

{\sc \'Agnes Backhausz.} E\"otv\"os Lor\'and University, Budapest, Hungary. {\tt agnes@cs.elte.hu}

{\sc Bal\'azs Szegedy.} University of Toronto, Canada. {\tt szegedyb@gmail.com} 

{\sc B\'alint Vir\'ag.} University of Toronto, Canada. {\tt balint@math.toronto.edu} 

\end{document}